\documentclass[11pt]{amsart}

\usepackage{amsmath,amscd,amssymb, epsfig,psfrag,subfigure,tikz}
\usepackage[colorlinks=true]{hyperref}
\usetikzlibrary{matrix}

\newtheorem{theorem}{Theorem}[section]

\newtheorem{prop}[theorem]{Proposition}
\newtheorem{cor}[theorem]{Corollary}
\theoremstyle{definition}
\newtheorem{definition}[theorem]{Definition}

\newtheorem{rmk}[theorem]{Remark}

   \RequirePackage{rotating}                   
    \def\HMt{%
       \setbox0=\hbox{$\widehat{\mathit{HM}}$}
       \setbox1=\hbox{$\mathit{HM}$}
       \dimen0=1.1\ht0
       \advance\dimen0 by 1.17\ht1
       \smash{\mskip2mu\raise\dimen0\rlap{%
          \begin{turn}{180}
              {$\widehat{\phantom{\mathit{HM}}}$}
           \end{turn}} \mskip-2mu    
                \mathit{HM}
    }{\vphantom{\widehat{\mathit{HM}}}}{}}

     \RequirePackage{rotating}                   
    \def\HMt{%
       \setbox0=\hbox{$\widehat{\mathit{HM}}$}
       \setbox1=\hbox{$\mathit{HM}$}
       \dimen0=1.1\ht0
       \advance\dimen0 by 1.17\ht1
       \smash{\mskip2mu\raise\dimen0\rlap{%
          \begin{turn}{180}
              {$\widehat{\phantom{\mathit{HM}}}$}
           \end{turn}} \mskip-2mu    
                \mathit{HM}
    }{\vphantom{\widehat{\mathit{HM}}}}{}}
    \newcommand{\HMf}{\widehat{\mathit{HM}}}

    \newcommand{\HMb}{\overline{\mathit{HM}}}
     \newcommand{\HMr}{{\mathit{HM}}}
    \newcommand{\spin}{\mathfrak{s}}
     \newcommand{\m}{\mathfrak{m}}
    \newcommand{\Z}{\mathbb{Z}}

\begin{document}
\title{Hyperbolic four-manifolds with vanishing Seiberg-Witten invariants}
\author[Ian Agol]{%
        Ian Agol} 
\address{%
    University of California, Berkeley \\
    970 Evans Hall \#3840 \\
    Berkeley, CA 94720-3840} 
\email{%
     ianagol@math.berkeley.edu}  
 \author[Francesco Lin]{%
        Francesco Lin} 
\address{%
    Princeton University, Department of Mathematics \\
    Fine Hall, Washington Road\\
Princeton NJ 08544-1000 } 
\email{%
     fl4@math.princeton.edu}  

\begin{abstract} 
We show the existence of hyperbolic 4-manifolds with vanishing Seiberg-Witten
invariants, addressing a conjecture of Claude LeBrun. This is achieved by showing, using results in geometric and arithmetic group theory, that certain hyperbolic $4$-manifolds contain $L$-spaces as hypersurfaces. 
\end{abstract}

\maketitle

\section{Introduction}
In \cite[Conjecture $1.1$]{LeBrun2002}, Claude LeBrun asked whether the Seiberg-Witten invariants of hyperbolic 4-manifolds vanish. This question stems from his result that for a hyperbolic $4$-manifold, Seiberg-Witten basic classes satisfy much stronger constraints than one would expect;  furthermore, it turns out to be related to several problems in low-dimensional topology \cite[\textsection 4]{Reid06}. Here, we show that there exist certain hyperbolic 4-manifolds with vanishing Seiberg-Witten invariants. 

\begin{theorem}\label{main}
There exist closed arithmetic hyperbolic 4-manifolds  with vanishing Seiberg-Witten invariants. 
\end{theorem}
In the statement, we consider all possible Seiberg-Witten invariants coming from evaluating elements of the cohomology ring $\Lambda^* H^1(X;\mathbb{Z})\otimes \Z[U]$ of the space of configurations. Theorem \ref{main} is proved by exhibiting hyperbolic $4$-manifolds admitting separating $L$-spaces, using the main result of \cite{KolpakovReidSlavich2018}; under mild additional conditions, this implies that such manifolds admit finite covers with vanishing Seiberg-Witten invariants. Our construction will show in fact that there are infinitely many commensurability classes of arithmetic hyperbolic $4$-manifolds containing representatives with vanishing Seiberg-Witten invariants. Furthermore, by interbreeding as in \cite{GPS}, one can also obtain non-arithmetic examples.

\vspace{0.3cm}
\textit{Acknowledgements.} The first author was funded by a Simons Investigator award. The second author was partially funded by NSF grant DMS-1807242. We thank Alan Reid for comments on an earlier draft.
\vspace{0.3cm}
\section{A vanishing criterion for the Seiberg-Witten invariants}
We discuss a vanishing result for the Seiberg-Witten invariants of four-manifolds containing a separating hypersurface. This is well-known to experts, but the exact form we will need is only implicitly stated in \cite{KM07}, so we will point it out for the reader's convenience. Most of our discussion is based on formal properties of the invariants, and we will follow closely follow the exposition of \cite[Chapter $3$]{KM07}.
\par
Consider a spin$^c$ structure $\spin_X$ on a closed oriented $4$-manifold $X$. For a cohomology class $u\in \Lambda^*H^1(Y;\mathbb{Z})\otimes\mathbb{Z}[U]$, we define the Seiberg-Witten invariant $\m(u|X,\spin_X)$ to be the evaluation of $u$ on the moduli space of solutions to the Seiberg-Witten equations. This is a topological invariant provided that $b_2^+\geq2$. The latter is not a restrictive assumption in our case; hyperbolic $4$-manifolds have signature zero by \cite[Theorem $3$]{Chern} and the Hirzebruch signature formula. Hence
\begin{equation*}
\chi(X)=2(1-b_1(X)+b_2^+(X)).
\end{equation*}
If $b_2^+(X)  \leq 1 $, we would have $\chi(X)\leq 4$; on the other hand, in all known examples of closed orientable hyperbolic $4$-manifolds $\chi\geq16$ \cite{MarstonMacLachlan2005, Long2008} (recall that by Gauss-Bonnet, volume and Euler characteristic are proportional).

We discuss a vanishing criterion for $\m(u|X,\spin_X)$. Let $Y$ be a closed, oriented three-manifold. To this, in \cite[Section 3.1]{KM07} it is defined for each spin$^c$ structure $\spin$ on $Y$ the monopole Floer homology groups fitting in the exact triangle of graded $\Z[U]$-modules
\begin{equation}\label{longexact}
\cdots\longrightarrow\HMb_{*}(Y,\spin) \stackrel{i_*}{\longrightarrow}  \HMt_{*}(Y,\spin)\stackrel{j_*}{\longrightarrow} \HMf_{*}(Y,\spin)\stackrel{p_*}{\longrightarrow} \HMb_{*}(Y,\spin)\longrightarrow\cdots
\end{equation}
where $U$ has degree $-2$ (notice that this convention differs from the one in the four-dimensional literature; this is because we identify $U$ with the corresponding capping operation in homology). The reduced Floer group $\HMr_*(Y,\spin)$ is defined to be the image of $j_*$ in $\HMf_{*}(Y,\spin)$ \cite[Definition 3.6.3]{KM07}.
We will be particularly interested in the case in which $Y$ is a rational homology sphere. In this case we have an identification of $\mathbb{Z}[U]$-modules (up to grading shift) with Laurent series \cite[Proposition 35.3.1]{KM07}
\begin{equation*}
\HMb_{*}(Y,\spin)\cong \Z[U^{-1},U].
\end{equation*}
\begin{definition}[\cite{KMOS}]
We say that a rational homology sphere $Y$ is an $L$-\textit{space} if, up to grading shift, $\HMf_*(Y,\spin)=\Z[U]$ as $\Z[U]$-modules for all spin$^c$ structures $\spin$.
\end{definition}
As the map $p_*$ in equation $(\ref{longexact})$ is an isomorphism in degrees low enough \cite[Section $22.2$]{KM07}, for an $L$-space $\HMr_*(Y,\spin)=0$ for all spin$^c$ structures $\spin$
\begin{prop} \label{vanishing}
Let $X$ be a four-manifold given as $X=X_1\cup_Y X_2$. Suppose that the separating hypersurface $Y$ is an $L$-space (so that in particular $b_1(Y)=0$), and that $b_2^+(X_i)\geq 1$. Then all the Seiberg-Witten invariants of $X$ vanish.
\end{prop}

\begin{rmk}
A simpler to state vanishing criterion is the following: if $b_1(X)=0$ and $b_2^+(X)$ is even, then all Seiberg-Witten invariants are zero. In fact, under this assumption all Seiberg-Witten moduli spaces are odd dimensional \cite[Theorem 1.4.4]{KM07}, while all classes in our cohomology ring are even dimensional. On the other hand, we are not aware of examples of hyperbolic $4$-manifolds satisfying these conditions. 
\end{rmk}

\begin{proof}[Proof of Proposition \ref{vanishing}]All we need to do is to discuss the results of \cite[Chapter $3$]{KM07} while keeping track of the specific spin$^c$ structures. First of all, notice that as $b_1(Y)=0$, a spin$^c$ structure $\spin_X$ on $X$ is determined by the restrictions $\spin_i=\spin_X\lvert_{X_i}$. This follows from the injectivity of the map $H^2(X;\Z)\rightarrow H^2(X_1;\Z)\oplus H^2(X_2;\Z)$ in the Mayer-Vietoris sequence, and the fact the these groups classify spin$^c$ structures. Let $\spin=\spin_X\lvert_Y$.
It is sufficient to show that $\m(u|X,\spin_X)=0$ for classes $u=u_1 u_2$ where $u_i$ is a cohomology class in the configuration space of $X_i$. Recall from \cite[Section 3.4]{KM07} that a cobordism $W$ from $Y_0$ to $Y_1$ induces a map in homology fitting with the exact triangle; furthermore, if $b_2^+(W)\geq 1$, we have that $\HMb_{*}(u|W,\spin)=0$ \cite[Proposition $3.5.2$]{KM07}. Given data as above, we can define the relative invariant $\psi_{(u_1|X_1,\spin_1)}\in \HMf_{*}(Y,\spin)$ obtained as follows: let $W_1$ be the cobordism obtained from $X_1$ by removing a ball, and consider the induced map
\begin{equation*}
\HMf_{*}(u_1|W_1,\spin_1):\HMf_{*}(S^3)=\Z[U]\rightarrow \HMf_{*}(Y,\spin).
\end{equation*}
Then $\psi_{(u_1|X_1,\spin_1)}=\HMf_{*}(u_1|W_1,\spin_1)(1)$. On the other hand, we have the commutative diagram
\begin{center}
\begin{tikzpicture}
\matrix (m) [matrix of math nodes,row sep=2.5em,column sep=2em,minimum width=2em]
  {
 \HMf_{*}(S^3)  & \HMb_{*}(S^3)\\
\HMf_{*}(Y,\spin)  & \HMb_{*}(Y,\spin)\\};
  \path[-stealth]
  (m-1-1) edge node [above]{$p_*$} (m-1-2)
   (m-2-1) edge node [above]{$p_*$} (m-2-2) 
  (m-1-1) edge node [left]{$\HMf_{*}(u_1|W_1,\spin_1)$} (m-2-1) 
    (m-1-2) edge node [right]{$\HMb_{*}(u_1|W_1,\spin_1)$} (m-2-2) 
  ;
\end{tikzpicture}
\end{center}
and as $b_2^+(W_1)\geq 1$, the vertical map on the right vanishes; in turn, this implies that $\psi_{(u_1|X_1,\spin_1)}\in\mathrm{ker}(p_*)=\HMr_*(Y,\spin)$. Similarly, using the map induced in cohomology by $W_2$, we obtain an element $\psi_{(u_2|X_2,\spin_2)}\in \HMr_{*}(-Y,\spin)$; this last group is by Poincar\'e duality identified with $\HMr^{*}(Y,\spin)$. The general gluing theorem in \cite[Equation $3.22$]{KM07}, when keeping track of the spin$^c$ structures, is then
\begin{equation*}
\m(u|X,\spin_X)=\langle \psi_{(u_1|X_1,\spin_1)},\psi_{(u_2|X_2,\spin_2)}\rangle,
\end{equation*}
where the angular brackets denote the natural pairing
\begin{equation*}
\HMr_{*}(Y,\spin)\times \HMr^{*}(Y,\spin)\rightarrow \Z.
\end{equation*}
In our assumptions, the group $\HMr_{*}(Y,\spin)$ vanishes, so this pairing is zero, and the result follows.
\end{proof}
\begin{rmk}\label{ztwocoeff}
In fact, for our purposes of understanding gluing formula for Seiberg-Witten invariants, it suffices to consider the reduced invariants with rational coefficients $\HMr_*(Y,\spin;\mathbb{Q})$. In particular, the previous discussion only relies on the vanishing of this group. Furthermore, via the universal coefficients theorem, this is implied by the vanishing of $\HMr_*(Y,\spin;\mathbb{Z}/2\mathbb{Z})$, so that our main result actually applies for the reduced Floer homology group with $\mathbb{Z}/2\mathbb{Z}$-coefficients.
\end{rmk}

Our examples will be based on the following.
\begin{cor} \label{nonseparating}
Suppose $X$ is a $4$-manifold with $b_2^+\geq1$ which admits an embedded non-separating $L$-space $Y$. Then $X$ admits infinitely many covers which have all vanishing Seiberg-Witten invariants.
\end{cor}
\begin{proof}
Consider the double cover $\tilde{X}$ of $X$ formed by gluing together two copies $W_1$ and $W_2$ of the cobordism from $Y$ to $Y$ obtained by cutting $X$ along $Y$, see Figure \ref{coverdouble}. Consider a properly embedded path $\gamma\subset W_1$ between the two copies of $Y$, and denote by $T$ its tubular neighborhood. We then have the decomposition $X=(W_1\setminus T)\cup(W_2\setminus T)$, where the two manifolds are glued along a copy of $Y\# \overline{Y}$; here $\overline{Y}$ denotes $Y$ with the opposite orientation. The latter is an $L$-space \cite[Section $4$]{Lin17}, and both $W_1\setminus T$ and $W_2\setminus T$ have $b_2^+\geq1$, so we conclude. Of course, we can modify this construction to provide infinitely many examples.
\end{proof}

\begin{figure}
  \centering
\def\svgwidth{\textwidth}
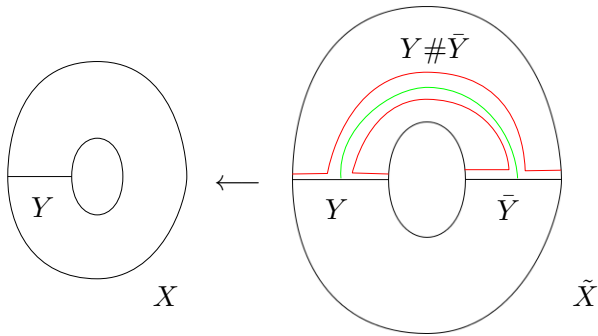
\caption{A double cover of $X$ contains a separating $L$-space.}
\label{coverdouble}
\end{figure}

\vspace{0.3cm}
\section{Geodesic hypersurfaces in arithmetic hyperbolic 4-manifolds}

In this section, we will discuss various properties of arithmetic hyperbolic lattices.
For the general case of arithmetic lattices, see \cite{Witte2015}, and for the
3-dimensional case, consult \cite{MR03}. We first review the
definitions and construction of arithmetic manifolds of simplest type. 

\begin{definition}
Let $G$ be a group, $H_1, H_2  \leq G$ be subgroups.
We say that $H_1$ is {\it commensurable in $G$} with $H_2$ if $[H_1:H_1\cap H_2] < \infty, [H_2:H_1\cap H_2] <\infty$. 
\end{definition} 
 
\begin{definition}
Consider a non-degenerate quadratic form $q:k^{n+1}\to k$ for   a totally real number field 
$k\subset \mathbb{R}$ with ring of integers $\mathcal{O}_k$. Assume that $q$ is Lorentzian, i.e. has signature $(n,1)$ over $\mathbb{R}$. 
Moreover, for each non-trivial embedding $\sigma: k\to \mathbb{R}$, assume
that $\sigma\circ q$ is positive definite. Let $O(q;k)$ denote the group
of matrices preserving $k$, i.e. linear transforms $A:k^{n+1}\to k^{n+1}$ such that $q\circ A = q$. 
Then the subgroup $O(q;\mathcal{O}_k) \subset O(q;k) \subset O(q;\mathbb{R})$
is a lattice, and acts discretely on the hyperboloid of two sheets $\mathcal{H}=\{ x\in \mathbb{R}^{n+1} | q(x) = -1\}$. Up to isometry, the group $O(q;\mathbb{R}) \cong O(n,1;\mathbb{R})$, the orthogonal
group associated to the quadratic form $-x_0^2+x_1^2+\cdots+x_n^2$. 
Projectivizing, $PO(q;\mathcal{O}_k)$ acts discretely on hyperbolic space $\mathbb{H}^n$,
which is the quotient of the hyperboloid $\mathcal{H}$ by the antipodal map. 
A hyperbolic orbifold $\mathbb{H}^n/\Gamma$ is said to be of {\it simplest type} if $\Gamma$ is commensurable (up to conjugacy)
with $PO(q;\mathcal{O}_k)$ for some such $q$.
\end{definition}

Example: Let $q_n: k^{n+1} \to k$ be defined by $q_n(x_0,x_1,\ldots, x_n)= -\sqrt{2}x_0^2+x_1^2+\cdots+x_n^2$ over the field $k=\mathbb{Q}(\sqrt{2})$. Let $\sigma :k\to k$ be the Galois automorphism
induced by $\sigma(\sqrt{2})= -\sqrt{2}$. Then $\sigma\circ q_n (x_0,\ldots,x_n) =\sqrt{2}x_0^2+x_1^2+\cdots+x_n^2$
is positive definite. Hence $PO(q_n; \mathbb{Z}[\sqrt{2}])$ is a discrete arithmetic lattice
acting on $\mathbb{H}^n$. See \cite[\textsection 6.4]{Witte2015}.

\begin{definition}
Let $G$ be a group. Then $G^{(2)} = \langle g^2 | g\in G\rangle$. \end{definition}

If $G$ is finitely
generated, then $G^{(2)}$ is finite-index in $G$, and $G/G^{(2)}$ is an elementary
abelian 2-group.

\begin{theorem}\label{embedding} Let $M^3$ be an orientable hyperbolic arithmetic 3-manifold of simplest type with $H_1(M;\mathbb{Z}/2) = 0$ and not defined over $\mathbb{Q}$. 
Then $M$ embeds as a totally geodesic non-separating submanifold in a compact arithmetic hyperbolic 4-manifold. 
\end{theorem}
\begin{proof}
Let $\Gamma = \pi_1(M) \leq Isom^+(\mathbb{H}^3)$. Since $M$ is a $\mathbb{Z}/2\mathbb{Z}$-homology sphere, $\Gamma^{(2)}=\Gamma$.  By \cite[Theorem 1.1 (2)]{KolpakovReidSlavich2018},
$\mathbb{H}^n/\Gamma^{(2)}\cong M$ embeds as a totally geodesic submanifold of a closed orientable hyperbolic 4-manifold $W$ (the fact that $M$ is not defined over $\mathbb{Q}$ implies that $W$ is compact). Briefly, this is proved by showing that $\Gamma^{(2)}\leq PO(q;k)$ so that it is commensurable with $PO(q;\mathcal{O}_k)$ for some Lorentzian quadratic form $q: k^4\to k$. Taking the quadratic form $Q_d=dy^2+q, d \in \mathbb{N}$,  we get an embedding of $PO(q;\mathcal{O}_k) < PO(Q_d;\mathcal{O}_k) < PO(Q_d;\mathbb{R})\cong PO(4,1;\mathbb{R})$. Then a subgroup separability result allows one to embed $\Gamma$ in a torsion-free lattice $\Lambda < PO(Q_d; k)$ so that $W=\mathbb{H}^4/\Lambda$. 
By \cite[Theorem 2]{Bergeron2000}, there exists a further finite-sheeted cover $\tilde{W}\to W$, and a lift $M\to \tilde{W}$ such that the lift of $M$ is non-separating in $\tilde{W}$. This is achieved again by a subgroup separability result. 
\end{proof}

\vspace{0.3cm}
\section{Examples}
The \textit{Fibonacci manifold} $M_n$ is the cyclic branched $n$-fold cover over the figure-eight knot. For $n=2$ we obtain a lens space, for $n=3$ the Hantzche-Wendt manifold, while for $n\geq 4$ it is hyperbolic.
\par
For every $n$ the Fibonacci manifold $M_n$ is an $L$-space. To see this, recall from \cite{VM96} that $M_n$ is the branched double cover over the closure of the $3$-braid $(\sigma_1\sigma_2^{-1})^n$ (see Figure \ref{branch}), which is alternating. Using the surgery exact triangle \cite{KMOS}, these can be shown to be $L$-spaces as in the context of Heegaard Floer homology \cite{OSz05}, with the caveat that in our setting the computation only holds with coefficients in $\mathbb{Z}/2\mathbb{Z}$; on the other hand this is enough for our purposes, see Remark \ref{ztwocoeff}. Notice also that for $n\neq 0$ modulo $3$, the closure is a knot, so that $M_n$ is a $\mathbb{Z}/2\mathbb{Z}$-homology sphere.

\begin{figure}
  \centering
\def\svgwidth{0.8\textwidth}
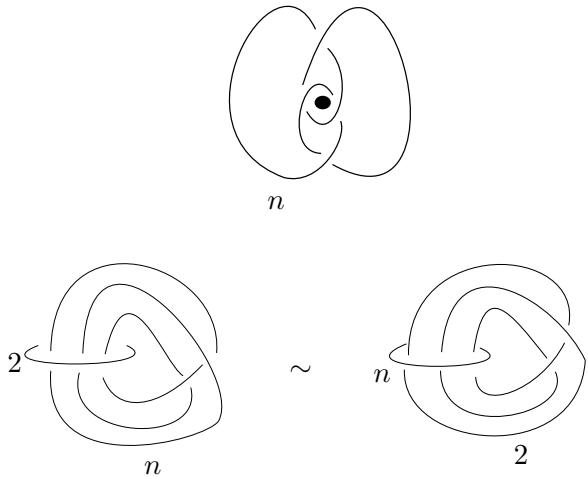
    \caption{In this picture, the numbers indicate the branching. The top picture has an obvious order $2$ rotational symmetry along the axis depicted by the big dot. The quotient is the link in $S^3$ depicted on the bottom left. This is isotopic to the link on the right (which is topologically the same, but with different branchings). Now, the curve with branching $2$ is the $3$-braid $\sigma_1\sigma_2^{-1}$, so that taking the $n$-fold branched cover along the other component we see that $M_n$ is the branched double cover over $(\sigma_1\sigma_2^{-1})^n$.}
    \label{branch}
\end{figure}

By \cite{HildenLozanoMontesinos1995}, $M_n$ is arithmetic when $n= 4, 5, 6, 8, 12$. Of these examples, $n=4,5,8$ are $\mathbb{Z}/2\mathbb{Z}$ homology spheres. The only one of these three which is simplest type and not defined over $\mathbb{Q}$ is $M_5$. This is example \cite[13.7.4(a)(iii)]{MR03}, which has invariant trace field a quartic field. As they point out, this is commensurable with a tetrahedral group \cite[13.7.4(a)(i)]{MR03} which is simplest type and not defined over $\mathbb{Q}$ by \cite[Theorem 1]{MacLachlanReid89}. It is defined over a quadratic form over the field $\mathbb{Q}(\sqrt{5})$.


Thus, by Theorem \ref{embedding}, $M_5$ has a non-separating embedding into a closed orientable hyperbolic 4-manifold $W$. 
We may assume that $\chi(W) > 2$ (by passing to a 2-fold cover if needed), and hence $b_2^+(W) >1$. Thus by Corollary \ref{nonseparating}, these embed into a hyperbolic 4-manifold with vanishing Seiberg-Witten invariants. This completes the proof of Theorem \ref{main}.

\begin{rmk}
One may also get other examples by cutting and doubling or using the interbreeding technique of Gromov-Piatetskii-Shapiro to get non-arithmetic examples. One can isometrically embed this $L$-space $M_5$ in infinitely many incommensurable hyperbolic 4-manifolds via the method of \cite{KolpakovReidSlavich2018} by taking the forms $Q_1$ and $Q_d$ in the proof of Theorem \ref{embedding} so that $d$ is square-free in $k=
\mathbb{Q}(\sqrt{5})$, and then cut and cross-glue to give a  closed non-arithmetic manifold containing $M_n$ as a non-separating hypersurface \cite[\textsection 2.9]{GPS}. 
\end{rmk}

\section{Conclusion}
We conclude by pointing out some natural questions related to our method.
\begin{enumerate}
\item Can one find an explicit hyperbolic example (such as the Davis manifold or the manifolds described in \cite{Long2008}) that satisfies the properties of Proposition \ref{vanishing}? Recall that the Davis manifold has $b_1=24$ and $b_2^+=36$ \cite{Davis}, so that all moduli spaces have odd dimension. 
\item Can one embed any orientable hyperbolic 3-manifold of simple type as a geodesic hypersurface in an orientable hyperbolic 4-manifold? More generally, can one show that orientable hyperbolic 3-manifolds have quasiconvex embeddings into orientable hyperbolic 4-manifolds? 
\item Can one use bordered Floer theory to compute the Seiberg-Witten invariants of Haken hyperbolic 4-manifolds (in the sense of \cite{FoozwellRubinstein})? 
\item Which commensurability classes of compact hyperbolic $3$-manifold of the simplest type contain $L$-spaces? Note that it is not even known if there are infinitely many commensurability classes of arithmetic rational homology 3-spheres. 
\end{enumerate}

\vspace{0.3cm}


\begin{thebibliography}{10}


\bibitem{Bergeron2000}
Nicolas Bergeron, \emph{Premier nombre de {B}etti et spectre du laplacien de
  certaines vari\'{e}t\'{e}s hyperboliques}, Enseign. Math. (2) \textbf{46}
  (2000), no.~1-2, 109--137.

\bibitem{Chern}
Chern, Shiing-Shen, \emph{On curvature and characteristic classes of a Riemann manifold}, Abh. Math. Sem. Univ. Hamburg 20 (1955), 117-126. 
  


\bibitem{FoozwellRubinstein}
Bell Foozwell and Hyam Rubinstein, \emph{Introduction to the theory of {H}aken
  {$n$}-manifolds}, Topology and geometry in dimension three, Contemp. Math.,
  vol. 560, Amer. Math. Soc., Providence, RI, 2011, pp.~71--84.
  
  \bibitem{GPS}
 Gromov, M. and Piatetski-Shapiro, I. \textit{Nonarithmetic groups in Lobachevsky spaces.} Inst. des Hautes \'Etudes Sci. Publ. Math. No. 66 (1988), 93-103. 

\bibitem{HildenLozanoMontesinos1995}
Hugh~M. Hilden, Mar\'{i}a~Teresa Lozano, and Jos\'{e}~Mar\'{i}a
  Montesinos-Amilibia, \emph{The arithmeticity of the figure eight knot orbifolds}, Topology '90 ({C}olumbus, {OH}, 1990) 169--183.
  
\bibitem{KolpakovReidSlavich2018}
Alexander Kolpakov, Alan~W. Reid, and Leone Slavich, \emph{Embedding arithmetic
  hyperbolic manifolds}, Mathematical Research Letters \textbf{25} (2018), no. 4, 1305--1328, \mbox{arXiv:1703.10561}.

\bibitem{KM07}
Kronheimer, Peter and Mrowka, Tomasz, \textit{Monopoles and three-manifolds}, New Mathematical Monographs, 10. Cambridge University Press, Cambridge, 2007. xii+796 pp

\bibitem{KMOS}
Kronheimer, P. and Mrowka, T. and Ozsv\'{a}th, P. and Szab\'{o}, Z., \textit{Monopoles and lens space surgeries}, Ann. of Math. (2) 165 (2007), no. 2, 457-546. 


\bibitem{LeBrun2002}
Claude LeBrun, \emph{Hyperbolic manifolds, harmonic forms, and
  {S}eiberg-{W}itten invariants}, Proceedings of the {E}uroconference on
  {P}artial {D}ifferential {E}quations and their {A}pplications to {G}eometry
  and {P}hysics ({C}astelvecchio {P}ascoli, 2000), vol.~91, 2002, pp.~137--154.


\bibitem{Lin17}
Lin, Francesco, \textit{{${\rm Pin}(2)$}-monopole {F}loer homology, higher compositions and connected sums}, J. Topol. 10 (2017), no. 4, 921-969. 

\bibitem{Long2008}
C.~Long, \emph{Small volume closed hyperbolic 4-manifolds}, Bull. Lond. Math.
  Soc. \textbf{40} (2008), no.~5, 913--916.

\bibitem{MacLachlanReid89}
C.~Maclachlan and A.~W. Reid, \emph{The arithmetic structure of tetrahedral
  groups of hyperbolic isometries}, Mathematika \textbf{36} (1989), no.~2,
  221--240 (1990).

\bibitem{MR03}
Colin Maclachlan and Alan~W. Reid, \emph{The arithmetic of hyperbolic
  3-manifolds}, Graduate Texts in Mathematics, vol. 219, Springer-Verlag, New
  York, 2003.
  
\bibitem{MarstonMacLachlan2005}
Marston Conder and Colin Maclachlan, \emph{Compact hyperbolic 4-manifolds of
  small volume}, Proc. Amer. Math. Soc. \textbf{133} (2005), no.~8, 2469--2476.
  
\bibitem{OSz05}
Ozsv\'{a}th, Peter and Szab\'{o}, Zolt\'{a}n, \textit{On the {H}eegaard {F}loer homology of branched double-covers}, Adv. Math. 194 (2005), no. 1, 1-33. 

\bibitem{Davis}
Ratcliffe, John G. and Tschantz, Steven T., \emph{On the {D}avis hyperbolic 4-manifold}, Topology Appl. 111 (2001), no. 3, 327-342. 


\bibitem{Reid87}
Alan~W. Reid, \emph{Arithmetic kleinian graups and their fuchsian subgroups},
  Ph.D. thesis, Aberdeen, 1987.

\bibitem{Reid06}
Alan~W. Reid, \emph{Surface subgroups of mapping class groups}, Problems on mapping class groups and related topics, 257-268, 
Proc. Sympos. Pure Math., 74, Amer. Math. Soc., Providence, RI, 2006.

\bibitem{VM96}
Vesnin, A. Yu. and Mednykh, A. D.,\textit{Fibonacci manifolds as two-sheeted coverings over a three-dimensional sphere, and the {M}eyerhoff-{N}eumann conjecture}, Siberian Math. J. 37 (1996), no. 3, 461-467

\bibitem{Witte2015}
Dave~Witte Morris, \emph{Introduction to arithmetic groups}, Deductive Press,
 http://deductivepress.ca/, 2015.


\end{thebibliography}

\end{document}